\documentclass[12pt]{amsart}
\usepackage[margin=1in]{geometry}

\usepackage{standalone}
\usepackage{enumerate}

\usepackage{tikz}
\usepackage{tikz-cd}
\usetikzlibrary{decorations.markings}
\tikzset{middlearrow/.style n args={4}{
		decoration={
			markings,
			mark=at position #1 with {\arrow{#2},\node[transform shape,#4] {#3};}},postaction={decorate}},
	middlearrow/.default={.5}{>}{}{below}	}
\usetikzlibrary{calc}

\usepackage{amssymb}
\usepackage{mathtools}
\usepackage{thmtools,thm-restate}

\usepackage{parskip}
\makeatletter
\def\thm@space@setup{%
		\thm@preskip=0.5\abovedisplayskip \thm@postskip=\belowdisplayskip
	}
\makeatother

\usepackage{etoolbox}
\AtEndEnvironment{thm}{\vspace{-0.5\baselineskip}}
\AtEndEnvironment{prop}{\vspace{-0.5\baselineskip}}
\AtEndEnvironment{lem}{\vspace{-0.5\baselineskip}}
\AtEndEnvironment{cor}{\vspace{-0.5\baselineskip}}
\AtEndEnvironment{rem}{\vspace{-0.5\baselineskip}}

\theoremstyle{plain}
\newtheorem{thm}{Theorem}

\newtheorem{cor}[thm]{Corollary}

\theoremstyle{definition}

\usepackage[all]{xy}

\usepackage{graphicx}
\graphicspath{{Images/}}
\usepackage{subcaption}

\usepackage{todonotes}

\usepackage[shortlabels]{enumitem}

\usepackage[british]{isodate}

\usepackage{bm}

\usepackage{hyperref}

\newcommand{\Z}{\mathbb{Z}}

\DeclareMathOperator{\Out}{Out}

\let\phi\varphi

\title{A hyperbolic free-by-cyclic group determined by its finite quotients}

\author{Naomi Andrew}
\address{Mathematical Institute, Andrew Wiles Building, Observatory Quarter, University of Oxford, Oxford OX2 6GG, United Kingdom}
\email{Naomi.Andrew@maths.ox.ac.uk}

\author{Paige Hillen}
\address{University of California - Santa Barbara Department of Mathematics}
\email{paigehillen@ucsb.edu}

\author{Robert Alonzo Lyman} 
\address{Department of Mathematics, Rutgers University - Newark}
\email{robbie.lyman@rutgers.edu}

\author{Catherine Eva Pfaff}
\address{Institute for Advanced Study \& Queen's University Department of Mathematics \& Statistics}
\email{catherine.pfaff@gmail.com}

\date{}

\begin{document}
	\begin{abstract}
		We show that the group $ \langle a,b,c,t : a^t=b,b^t=c,c^t=ca^{-1} \rangle$ is profinitely rigid amongst free-by-cyclic groups, providing the first example of a hyperbolic free-by-cyclic group with this property.
	\end{abstract}
	
	\maketitle	
	
	All groups considered will be finitely generated and residually finite. Two groups are said to be \emph{profinitely isomorphic} if they share the same set of finite quotients, and a group $G$ is \emph{profinitely rigid} (within a class $\mathcal{C}$) if any group (within $\mathcal{C}$) which is profinitely isomorphic to $G$ is isomorphic to $G$. One can ask which groups within which classes are profinitely rigid, though this question can be subtle even when the class is very restricted: finitely generated abelian groups are profinitely rigid, free groups are profinitely rigid among themselves but it is an open question due to Remeslennikov \cite[Question 5.48]{KourovkaNotebook} whether this holds among finitely generated groups, and there are pairs of non-isomorphic virtually free (even virtually $\Z$!) groups which are profinitely isomorphic \cite{BaumslagNotPR}. 
	
	Grothendieck asked if there are groups which are profinitely isomorphic to one of their proper, non-isomorphic subgroups: Platonov and Tavgen \cite{PlatonovTavgen1986} provide examples of this phenomenon for finitely generated groups, and Bridson and Grunewald \cite[Theorem 1.1]{BridsonGrunewald2004} for finitely presented groups. For a discussion of these facts, as well as a broader introduction to profinite rigidity of groups, see for instance \cite{ReidICMsurvey}.
	
	Here we consider profinite rigidity within the class of free-by-cyclic groups. A free-by-cyclic group is a semidirect product $G:=F_r \rtimes_\varphi \Z$,
  where $F_r$ denotes the free group of finite rank $r$. 
	If two automorphisms $\varphi, \psi$ of $F_r$ represent conjugate or conjugate inverse elements of $\Out(F_r)$ then they define isomorphic free-by-cyclic groups \cite[Lemma 2.1]{BMVrank2}. The converse holds when the abelianisation has rank 1, $b_1(G)=1$ \cite[Theorem 2.4]{BMVrank2}, though not in general \cite[p1678]{BMVrank2}. 
	
	We provide the first known example of a hyperbolic free-by-cyclic group that is profinitely rigid amongst free-by-cyclic groups.
	
	\begin{thm}
		\label{thm:main}
		The group $G \cong \langle a,b,c,t : a^t=b, b^t=c, c^t=ca^{-1} \rangle$ is profinitely rigid amongst free-by-cyclic groups.
	\end{thm}
	
	Hughes and Kudlinska \cite{HKprofinite} have shown that if $G$ has $b_1(G)=1$ then in many cases $G$ is \emph{almost} profinitely rigid: it shares a profinite isomorphism class with at most finitely many non-isomorphic free-by-cyclic groups. These cases include when the defining automorphism is irreducible (it has no non-trivial preserved free factor system) and when the rank of the kernel is 3. Building off work of Bridson, Reid and Wilton \cite{BridsonReidWilton2017} they also show that if the kernel has rank 2 then $G$ is profinitely rigid among free-by-cyclic groups. Note that in the rank 2 case $G$ is never hyperbolic.
	
	In a different direction Bridson and Piwek have very recently shown that free-by-cyclic groups with centre \cite[Theorem 1.1]{BridsonPiwek}, or where the cyclic group is instead required to be \emph{finite} cyclic \cite[Theorem 1.2]{BridsonPiwek}, are profinitely rigid among respectively all free-by-(infinite cyclic) groups and free-by-(finite cyclic) groups.
		
	Following the analogy between free-by-cyclic groups and three manifolds, we recall two cognate results on the profinite rigidity of three manifolds. First, the fundamental group of the figure-eight knot complement is profinitely rigid amongst all fundamental groups of three manifolds \cite[Theorem A]{BridsonReid2020}: the proof goes via the identification of this manifold with a once punctured torus bundle over the circle, which algebraically corresponds to some $F_2 \rtimes \Z$. Second, and stronger, the fundamental group of the Weeks manifold (the unique closed orientable hyperbolic three manifold of minimal volume) is profinitely rigid, with no need to restrict the class concerned \cite[Theorem 9.1]{BridsonMcReynoldsReidSpitler2020}.

\subsection*{Proof of the Theorem}	
	The proof will be an application of Hughes and Kudlinska's characterisation of properties detected by the profinite isomorphism class of a free-by-cyclic group (Theorem \ref{thm:hk}), together with Hillen's work (Corollary \ref{cor:paige}) controlling stretch factors of elements of $\Out(F_r)$:

\begin{thm}[Hughes--Kudlinska]
	\label{thm:hk}
	Suppose $G \cong F_r \rtimes \Z$ is a free-by-cyclic group that is hyperbolic and has $b_1(G)=1$. If $H \cong F_s \rtimes \Z$ is profinitely isomorphic to $G$, then \begin{itemize}
		\item \cite[Theorem C]{HKprofinite} $H$ is hyperbolic.
		\item \cite[Theorem B(1)]{HKprofinite} and \cite[Lemma 3.1]{BridsonReid2020} $r=s$.
		\item \cite[end of Theorem B]{HKprofinite} The set of stretch factors $\{\lambda_G^+, \lambda_G^-\}$ associated to the defining outer automorphism of $G$ (and its inverse) agrees with $\{\lambda_H^+, \lambda_H^-\}$.
	\end{itemize}
\end{thm}

We briefly recap the concepts introduced within the theorem. Brinkmann showed that $G$ is hyperbolic if and only if its defining automorphism is \emph{atoroidal:} it has no periodic conjugacy classes \cite[Theorem 1.2]{BrinkmannHyp}. The stretch factor $\lambda$ of an automorphism is defined to be \[ \sup_{w \in F_r} \limsup \sqrt[n]{\| \varphi^n(w) \|},\] where $\|w\|$ is the cyclically reduced word length of the element $w$; it records ``how fast'' elements grow under repeated application of the automorphism, and is an $\Out(F_r)$-conjugacy class invariant. Elements of $\Out(F_r)$ can and often do have different stretch factors from their inverses, so we must record both. Note that the set $\{\lambda_H^+, \lambda_H^-\}$ is well defined since the profinite isomorphism class of a group determines its abelianisation (see for instance \cite[Remark 3.2]{ReidICMsurvey}), and in particular $H$ must also have $b_1(H)=1$ and so there is a unique (up to inverting and $\Out(F_r)$-conjugacy) defining automorphism for $H$.
	
\begin{cor}~{\cite[Corollary 8.1]{PaigeSymmetries}}
	\label{cor:paige}
	The element $\psi$ of $\Out(F_3)$ defined by sending $a \mapsto b, b \mapsto c, c \mapsto ca^{-1}$ defines the unique $\Out(F_3)$-conjugacy class of infinite order irreducible elements realising the minimal stretch factor $\lambda \approx 1.167$.
\end{cor}

\begin{proof}[Proof of Theorem~\ref{thm:main}]
First we check that $G$ satisfies the hypotheses of Theorem \ref{thm:hk}. A quick computation with the abelianisation verifies that $b_1(G)=1$, while hyperbolicity follows from Brinkmann's theorem  \cite[Theorem 1.2]{BrinkmannHyp} since the single fold representative has no periodic Nielsen paths so there cannot be a periodic conjugacy class.

Now, suppose $H$ is another free-by-cyclic group that is profinitely isomorphic to $G$. From Theorem~\ref{thm:hk} we know that $H$ is some $F_3 \rtimes \Z$, that it is hyperbolic, and that the defining automorphism and its inverse will have the same stretch factor(s) as $\psi$ and its inverse; in particular one of them must be $\lambda$. Let $\varphi$ be the choice with smaller stretch factor.

It follows from Brinkmann's theorem that $\varphi$ is atoroidal. We also observe, as for instance in the proof of \cite[Corollary F]{HKprofinite}, that if an element of $\Out(F_3)$ is atoroidal then it must be irreducible (the point is that a preserved free factor would have rank 1 or 2, and in either case there is a periodic conjugacy class). 
But then, by Corollary~\ref{cor:paige}, we have that $\varphi$ and $\psi$ are conjugate as elements of $\Out(F_3)$, and so they define isomorphic free-by-cyclic groups.
\end{proof}

\subsection*{Acknowledgements}
We thank Sam Hughes for suggesting the application of our work to profinite rigidity, and for comments on an earlier draft. We are grateful to the Women in Groups, Geometry and Dynamics (WiGGD) program, from which this collaboration arose. This work has received funding from the European Research Council (ERC) under the European Union's Horizon 2020 research and innovation programme (Grant agreement No. 850930), the Royal Society of Great Britain, and from an NSERC Discovery Grant. The fourth author is grateful to the Institute for Advanced Study for their hospitality and Bob Moses for funding her membership. The authors certify that there is no competing interest. 

\bibliographystyle{alpha}
\bibliography{biblio}

\end{document}